\documentclass[final,3p,times,authoryear]{elsarticle}

\usepackage{amssymb}
\usepackage{amsmath}

\usepackage{natbib}
\setcitestyle{numbers,square}
\usepackage{bbm}
\usepackage{color}
\usepackage{amsmath}
\usepackage{amsthm}
\usepackage{amssymb}
\usepackage{enumerate}

\usepackage{graphicx}
\usepackage{xcolor}
\usepackage{amsthm}
\usepackage{enumerate}

\newtheorem{Theorem}{Theorem}[section]
\newtheorem{Corollary}[Theorem]{Corollary}

\begin{document}

\begin{frontmatter}



\title{Boundary behavior of continuous-state interacting multi-type branching
processes with immigration} 


\author[label1]{Peng Jin}

\affiliation[label1]{organization={Guangdong Provincial/Zhuhai Key Laboratory of IRADS, and Department of Mathematical Sciences, Beijing
Normal-Hong Kong Baptist University},
            city={Zhuhai 519087}, 
            country={China}}
\ead{pengjin@bnbu.edu.cn}

\author[label2,label3]{Jiaqi Zhou}
\affiliation[label2]{organization={Hong Kong Baptist University},
             country={Hong Kong} } 
             
\affiliation[label3]{organization={ Beijing
Normal-Hong Kong Baptist University},
             city={Zhuhai 519087},             
             country={China}}
\ead{s230209704@mail.bnbu.edu.cn}         
\begin{abstract}
In this paper, we study continuous-state interacting multi-type branching
processes with immigration (CIMBI processes), where inter-specific
interactions---whether competitive, cooperative, or of a mixed type---are
proportional to the product of their type-population masses. We establish
sufficient conditions for the CIMBI process to never hit the boundary
$\partial\mathbb{R}_{+}^{d}$ when starting from the interior of $\mathbb{R}_{+}^{d}$.
Additionally, we present two results concerning boundary attainment.
In the first, we consider the diffusion case and prove that when the
constant immigration rate is small and diffusion noise is present
in each direction, the CIMBI process will almost surely hit the boundary
$\partial\mathbb{R}_{+}^{d}$. In the second result, under similar
conditions on the constant immigration rate and diffusion noise, but
with jumps of finite activity, we show that the CIMBI process 
hits the boundary $\partial\mathbb{R}_{+}^{d}$ with positive probability.
\end{abstract}

\begin{keyword}
Multi-type continuous-state branching process \sep immigration \sep interaction \sep boundary behavior \sep extinction

\MSC 60J80 \sep  60H20

\end{keyword}

\end{frontmatter}



\section{Introduction}

\numberwithin{equation}{section}

Continuous-state branching processes with immigration, CBI processes
for short, were introduced in \citep{MR290475,MR279902} and they
arise naturally as scaling limits of Galton-Watson branching processes
with immigration as shown in \citep{MR2225068}. \citep{MR2134113}
first considered a logistic branching process featuring intra-specific
competition with a quadratic death rate. More general continuous-state
branching processes with competition were introduced in \citep{MR3877546}
and have received much attention in recent years. For the single-type
case, \citep{MR3414448,MR4383204,MR4139087,MR4332223} investigated
the extinction probability, while explosion was studied in \citep{MR3940763}
and also in \citep{ma2024explosion} for similar models in a Lévy
environment. \citep{MR3983343} introduced a class of nonlinear branching
processes and provided criteria for coming down from infinity, extinction
and explosion respectively, see also \citet{MR3983343,MR4236676}.
Very recently, the strong Feller property and exponential convergence
rate to quasi-stationary distribution were proved in \citep{MR4790450}
for CB processes with competition that is strong enough near infinity,
while sufficient conditions for exponential ergodicity of CBI processes
with competition were established in \citet{MR4863048}.

In contrast with the single-type case, works concerning multi-type
continuous-state branching processes with competition or more general
interaction mechanisms have been sparse so far. \citet{MR3208120}
first introduced a two-type CBI process with intra-specific competition
as the unique strong solution of a SDE with jumps and established
a comparison principle. \citet{cattiaux2010competitive} considered
a two-dimensional stochastic Lotka-Volterra system with intra-specific
competition and inter-specific cooperation or competition and studied
its long-time behavior conditioned on non-extinction. In \citet{Fittipaldi16082025},
continuous-state interacting multi-type branching processes (CIMBP
for short) with full generality were studied and, among many other
things, the authors proved that CIMBP can be obtained by a Lamperti-type
transformation of multi-dimensional Lévy processes.

In this paper, we introduce and study continuous-state interacting
multi-type branching processes with immigration (CIMBI processes),
adding the effect of immigration to CIMBP model introduced in \citet{Fittipaldi16082025}.
From another point of view, a CIMBI process can also be obtained by
adding interaction to a general multi-type CBI process. We focus on
the construction of a CIMBI process as the unique strong solution
of a SDE with jumps in $\mathbb{R}_{+}^{d}$ and also investigate
its boundary behavior at $\partial\mathbb{R}_{+}^{d}$. We are interested
in three types of boundary behaviors: 1) the process never hits the
boundary $\partial\mathbb{R}_{+}^{d}$; 2) the process hits the boundary
$\partial\mathbb{R}_{+}^{d}$ almost surely; 3) the process hits $\partial\mathbb{R}_{+}^{d}$
with positive probability. For each behavior we provide some easy
to check sufficient conditions, see Section 3 for details.

The boundary behavior of one-type CB (or CBI in general) processes
with or without competition can be studied using martingale methods
or Lyapunov function technique and has already been addressed in many
previous works, see \citet{G74,CPU13,DFM14,FU14,MR2134113,MR3414448,MR3980150,MR3983343,MR4139087,MR4236676}
and the references therein. Extending these results to multi-dimensions
is not straightforward, since the powerful martingale methods don't
seem to work in this case. Nevertheless, based on a comparison principle,
sufficient conditions for boundary non-attainment of multi-type CBI
processes were given in \citet{MR4195178}. The extinction time for
multi-type CB processes was recently investigated in \citet{MR4575008}
using a Lamperti-type transformation. Unfortunately the methods in
\citet{MR4195178,MR4575008} don't apply to general CIMBI processes.
To study the boundary behavior of CIMBI processes introduced in this
paper, we use the Foster-Lyapunov type criteria from \citet{MR4419569}
for boundary non-attainment, while for boundary attainment we use
some ideas from \citet{cattiaux2010competitive} where a stochastic
Lotka-Volterra system on $\mathbb{R}_{+}^{2}$ is transformed to a
Kolmogorov diffusion.

The structure of this paper is organized as follows. In Section 2,
we introduce our model and prove the existence of a unique strong
solution to the corresponding SDE. In Section 3.1, the behavior of
not hitting the boundary is studied by a Foster-Lyapunov type criteria
in \citep[Proposition 2.1]{MR4419569}. In Section 3.2, we first provide
some sufficient conditions for a CIMBI diffusion to hit the boundary
almost surely. Then we extend it to the case of finite Lévy measures
in the branching and immigration mechanisms; however, as a compensation,
in this case we are only able to show that the boundary is hit with
positive probability. At the end of Section 3.2, we show that for
competitive interaction, the conditions for boundary attainment can
be slightly relaxed.

\section{Preliminaries}

Let $\left(\Omega,\mathcal{F},(\mathcal{F}_{t})_{t\ge0},\mathbb{P}\right)$
be a filtered probability space satisfying the usual hypotheses on
which the following objects are defined:
\begin{itemize}
\item $W(t)=(W_{1}(t),...,W_{d}(t))$ is a standard $d$-dimensional $\mathcal{F}_{t}$-Brownian
motion;
\item $N_{i}\left(ds,du,dz\right)$ is a Poisson random measure on $\mathbb{R}_{+}\times\mathbb{R}_{+}\times\mathbb{R}_{+}^{d}$
with intensity $dsdu\mu_{i}\left(dz\right)$ and compensated random
measure $\tilde{N}_{i}\left(ds,du,dz\right)=N_{i}\left(ds,du,dz\right)-dsdu\mu_{i}\left(dz\right)$,
where $\mu_{i}$ is a Borel measure on $\mathbb{R}_{+}^{d}$ satisfying
\[
\int_{\mathbb{R}_{+}^{d}}\left(z_{i}\wedge z_{i}^{2}+\sum_{j\in\{1,2,...,d\}\backslash\{i\}}z_{j}\right)\mu_{i}(dz)<\infty,\quad\mu_{i}(\{0\})=0,\quad\text{ for }i\in\{1,2,...,d\}.
\]
Additionally, for each $U\in\mathcal{B}(\mathbb{R}_{+}\times\mathbb{R}_{+}^{d})$,
$N_{i}\left((0,t]\times U\right)$ is $\mathcal{F}_{t}$-adapted and
if $t>s\ge0$, then $N_{i}\left((s,t]\times U\right)$ is independent
of $\mathcal{F}_{s}$;
\item $N_{0}(ds,dz)$ is a Poisson random measure on $\mathbb{R}_{+}\times\mathbb{R}_{+}^{d}$
with intensity $ds\nu\left(dz\right)$ and compensated random measure
$\tilde{N}_{0}\left(ds,dz\right)=N_{0}\left(ds,dz\right)-ds\nu\left(dz\right)$,
where $\nu$ is a Borel measure on $\mathbb{R}_{+}^{d}$ satisfying
$\int_{\mathbb{R}_{+}^{d}}|z|\nu(dz)<\infty$ and $\nu(\{0\})=0$.
Additionally, for each $U\in\mathcal{B}(\mathbb{R}_{+}^{d})$, $N_{0}\left((0,t]\times U\right)$
is $\mathcal{F}_{t}$-adapted and if $t>s\ge0$, then $N_{0}\left((s,t]\times U\right)$
is independent of $\mathcal{F}_{s}$.
\item $W,N_{0},N_{1},\ldots,N_{d}$ are mutually independent.
\end{itemize}
Consider the following stochastic differential equation with jumps
in $\mathbb{R}_{+}^{d}$: for $i\in\{1,...,d\}$,
\begin{equation}
\begin{aligned}X_{i}(t)= & x_{i}+\int_{0}^{t}\bigg(\eta_{i}+\sum_{j=1}^{d}b_{ij}X_{j}(s)+\gamma_{i}\left(X(s)\right)\bigg)ds+\int_{0}^{t}\sqrt{2\sigma_{i}X_{i}(s)}dW_{i}(s)+\int_{0}^{t}\int_{\mathbb{R}_{+}^{d}}z_{i}N_{0}\left(ds,dz\right)\\
 & +\int_{0}^{t}\int_{\mathbb{R}_{+}}\int_{\mathbb{R}_{+}^{d}}z_{i}\mathbbm{1}_{\{u\leq X_{i}(s-)\}}\tilde{N}_{i}\left(ds,du,dz\right)+\sum_{j\neq i}\int_{0}^{t}\int_{\mathbb{R}_{+}}\int_{\mathbb{R}_{+}^{d}}z_{i}\mathbbm{1}_{\{u\leq X_{j}(s-)\}}N_{j}\left(ds,du,dz\right),
\end{aligned}
\label{ndimCBIC}
\end{equation}
where
\begin{itemize}
\item $x=(x_{1},\ldots,x_{d}),\,\eta:=(\eta_{1},...,\eta_{d})\text{ and }\sigma:=(\sigma{}_{1},...,\sigma_{d})\text{ all belong to }\mathbb{R}_{+}^{d}$;
\item $B:=(b_{ij})_{i,j\in\{1,2,...,d\}}$ is such that $b_{ij}\geq0$ for
$j\neq i$ and $\gamma(x)=(\gamma_{1}(x),...,\gamma_{d}(x))$ with
$\gamma_{i}(x)=\sum_{j=1}^{d}c_{ij}x_{i}x_{j}$, where $c_{ij}$ are
constants and $c_{ii}<0$ for $i=1,\ldots,d$.
\end{itemize}
The matrix $c:=(c_{ij})_{1\le i,j\le d}$ is the so-called interaction
matrix. Note that $c_{ii}<0$ indicates intra-specific competitions.
Depending on the signs of the off-diagonal entries of the interaction
matrix, there are three regimes of interactions:
\begin{enumerate}
\item competition: $c_{ij}\le0$ for all $i\neq j$;
\item cooperation: $c_{ij}\ge0$ for all $i\neq j$;
\item mixed-type: there exist $c_{ij}$ and $c_{\tilde{i}\tilde{j}}$ with
$i\neq j,\tilde{i}\neq\tilde{j}$ and $c_{ij}c_{\tilde{i}\tilde{j}}<0$.
\end{enumerate}
Following similar arguments as in \citep{Fuzongfei2010spa,MR3208120,Fittipaldi16082025},
we can obtain the following theorem. \begin{Theorem}\label{pathwise unique strong solution}
Assume that
\begin{align}
\sum_{i=1}^{d}\gamma_{i}(x)\le0,\quad x\in\mathbb{R}_{+}^{d}.\label{interaction condition}
\end{align}
Then the SDE (\ref{ndimCBIC}) has a unique $\mathbb{R}_{+}^{d}$-valued
strong solution. \end{Theorem}
\begin{proof}
The pathwise uniqueness can be obtained using the same proof of \citep[Theorem 3.2]{Fittipaldi16082025}.
For $m\in\mathbb{N}$, consider the truncated equation
\begin{align}
X_{i}^{m}(t) & =x_{i}+\int_{0}^{t}\bigg(\eta_{i}+\sum_{j=1}^{d}b_{ij}\cdot X_{j}^{m}(s)\wedge m+\gamma_{i}\left(X^{m}(s)\wedge m\right)-\alpha_{i}(m)\cdot X_{i}^{m}(s)\wedge m\bigg)ds+\int_{0}^{t}\sqrt{2\sigma_{i}\cdot X_{i}^{m}(s)\wedge m}dW_{i}(s)\nonumber \\
 & +\int_{0}^{t}\int_{\mathbb{R}_{+}^{d}}z_{i}\wedge mN_{0}\left(ds,dz\right)+\int_{0}^{t}\int_{\mathbb{R}_{+}}\int_{\mathbb{R}_{+}^{d}}z_{i}\wedge m\mathbbm{1}_{\{u\leq X_{i}^{m}(s-)\wedge m\}}\tilde{N}_{i}\left(ds,du,dz\right)\nonumber \\
 & +\sum_{j\neq i}\int_{0}^{t}\int_{\mathbb{R}_{+}}\int_{\mathbb{R}_{+}^{d}}z_{i}\wedge m\mathbbm{1}_{\{u\leq X_{j}^{m}(s-)\wedge m\}}N_{j}\left(ds,du,dz\right),\label{eq: truncated sde}
\end{align}
where $\alpha_{i}(m)=\int_{\mathbb{R}_{+}^{d}}\big(z_{i}-z_{i}\wedge m\big)\mu_{i}(dz)$
and $X^{m}(t)\wedge m=(X_{1}^{m}(t)\wedge m,...,X_{d}^{m}(t)\wedge m)$.
Then \eqref{eq: truncated sde} has a unique $\mathbb{R}_{+}^{d}$-valued
strong solution, using a similar argument to \citep[Theorem 3.1]{Fittipaldi16082025}
(see also \citep[Theorem 3.2]{MR3208120}). Then repeating the arguments
in \citep[Proposition 2.4]{Fuzongfei2010spa}, $X^{m}(t)$ has the
same path as $X^{n}(t)$ on $[0,\tau_{m})$ if $n\ge m$, where $\tau_{m}:=\inf\{t\geq0:\max_{i\in\{1,...,d\}}X_{i}(t)\geq m\}$,
and existence of a strong solution to (\ref{ndimCBIC}) follows if
we can show that $\tau_{m}\nearrow+\infty$ a.s. as $m\to\infty$.
Thanks to condition (\ref{interaction condition}), we can find a
constant $K>0$ such that
\begin{align*}
\mathbb{E}\bigg[1+\sum_{i=1}^{d}X_{i}^{m}(t\wedge\tau_{m})\bigg]\leq & 1+\sum_{i=1}^{d}x_{i}+\mathbb{E}\bigg[\int_{0}^{t}2K+2K\sum_{i=1}^{d}X_{i}^{m}(s\wedge\tau_{m})ds\bigg].
\end{align*}
By Gronwall's inequality,
\[
\mathbb{E}[1+\sum_{i=1}^{d}X_{i}^{m}(t\wedge\tau_{m})]\leq\left(1+\sum_{i=1}^{d}x_{i}\right)\exp2Kt.
\]
Noting that $\sum_{i=1}^{d}X_{i}^{m}(\tau_{m})\ge m$ by right-continuity
of sample paths, we obtain, for each $t\ge0$,
\[
(1+m)\mathbb{P}\{\tau_{m}\leq t\}\leq\left(1+\sum_{i=1}^{d}x_{i}\right)\exp{2Kt},
\]
which implies that $\tau_{m}\nearrow+\infty$ a.s as $m\to\infty$.
\end{proof}

\section{Behavior at the boundary $\mathbb{R}_{+}^{d}$}

Throughout this section we assume that condition \eqref{interaction condition}
is true.

\subsection{Boundary non-attainment}

\begin{Theorem} Consider the SDE (\ref{ndimCBIC}). Suppose that
$x_{i}>0$ for $i=1,\ldots,d$. If $\eta_{i}>\sigma_{i}$ for each
$i\in\{1,...,d\}$, then
\begin{align}
\mathbb{P}[X_{i}(t)>0,\enspace\forall t>0]=1,\enspace\text{for each }i\in\{1,...,d\}.\label{nonattainment}
\end{align}
In addition, if $\eta_{i}=\sigma_{i}$ and
\[
\int_{|z|\leq1}z_{i}\mu_{i}(dz)<\infty
\]
for each $i\in\{1,...,d\}$, then assertion (\ref{nonattainment})
also holds. \end{Theorem}
\begin{proof}
Our proof is essentially based on the $d$-dimensional
analogue of \citep[Proposition 2.1]{MR4419569}. We will find a function
$0\le f\in C^{2}\big((0,\infty)^{d}\big)$ satisfying : 1) $\lim_{x_{1}\wedge...\wedge x_{d}\to0^{+}}f(x)=+\infty$;
and 2) $\mathcal{L}f(x)\leq d_{m}f(x)$ for all $x\in(0,m)^{d}$ and
large $m\ge1$, where $\mathcal{L}$ is the generator of $X(t)$ and
$d_{m}>0$ are constants. Then we repeat the arguments in the proof
of \citep[Proposition 2.1]{MR4419569} to get (\ref{nonattainment}).

Let $f(x):=1+\sum_{i=1}^{d}(x_{i}-\ln x_{i}),x\in(0,\infty)^{d}$.
Then for $x\in(0,\infty)^{d}$, $\mathcal{L}f(x)$ is well-defined
and $\mathcal{L}f(x)=\mathcal{L}_{1}f(x)+\mathcal{L}_{2}f(x)$, where
\begin{align*}
 & \mathcal{L}_{1}f(x)=\sum_{i=1}^{d}(\eta_{i}+\frac{\sigma_{i}-\eta_{i}}{x_{i}})+\sum_{i,j=1}^{d}b_{ij}(x_{j}-\frac{x_{j}}{x_{i}})+\sum_{i,j=1}^{d}c_{ij}x_{j}(x_{i}-1)\\
 & \mathcal{L}_{2}f(x)=\int_{\mathbb{R}_{+}^{d}}\sum_{i=1}^{d}\bigg(z_{i}+\ln\frac{x_{i}}{x_{i}+z_{i}}\bigg)\nu(dz)+\sum_{i=1}^{d}x_{i}\int_{\mathbb{R}_{+}^{d}}\left(\sum_{j=1}^{d}\bigg(z_{j}+\ln\frac{x_{j}}{x_{j}+z_{j}}\bigg)-z_{i}+\frac{z_{i}}{x_{i}}\right)\mu_{i}(dz).
\end{align*}
Obviously $f$ satisfies (1). We next show that $f$ also satisfies
(2). 

Suppose $\eta_{i}>\sigma_{i}$. We can then find a small constant
$\epsilon\in(0,1)$ such that
\begin{align}
\eta_{i}\geq\sigma_{i}+\frac{1}{2}\int_{|z|\leq\epsilon}z_{i}^{2}\mu_{i}(dz)\label{immigration condition}
\end{align}
for each $i\in\{1,...,d\}$. For $x\in(0,m)^{d}$, we have
\[
\mathcal{L}_{1}f(x)\leq\sum_{i=1}^{d}\bigg(\eta_{i}+\frac{\sigma_{i}-\eta_{i}}{x_{i}}+|b_{ii}|+\sum_{j=1}^{d}\bigg(|b_{ij}|+(1+m)|c_{ij}|\bigg)x_{j}\bigg).
\]
Note that $|t-\ln(1+t)|\le t^{2}/2$ for all $t\ge0.$ Then we obtain
\[
x_{i}\int_{|z|\leq\epsilon}\bigg(\frac{z_{i}}{x_{i}}+\ln\frac{x_{i}}{x_{i}+z_{i}}\bigg)\mu_{i}(dz)=x_{i}\int_{|z|\leq\epsilon}\bigg(\frac{z_{i}}{x_{i}}-\ln\left(1+\frac{z_{i}}{x_{i}}\right)\bigg)\mu_{i}(dz)\leq\frac{1}{2x_{i}}\int_{|z|\leq\epsilon}z_{i}^{2}\mu_{i}(dz),
\]
and
\[
x_{i}\int_{|z|>\epsilon}\big(\frac{z_{i}}{x_{i}}+\ln\frac{x_{i}}{x_{i}+z_{i}}\big)\mu_{i}(dz)\leq\int_{|z|>\epsilon}z_{i}\mu_{i}(dz).
\]
So $\mathcal{L}f(x)\leq C_{m}\big(1+\sum_{i=1}^{d}x_{i}\big)$ for
some large constant $C_{m}$ by inequality (\ref{immigration condition}).
Noticing that $e(e-1)^{-1}(x-\ln x)\geq x$, $x>0$,
it is easy to find another constant $d_{m}>0$ such that $\mathcal{L}f(x)\leq d_{m}f(x)$,
$x\in(0,m)^{d}$. Hence condition (2) is true. 

For the case $\eta_{i}=\sigma_{i}$ and $\int_{|z|\leq1}z_{i}\mu_{i}(dz)<\infty$,
we have
\[
\mathcal{L}_{2}f(x)\leq\int_{\mathbb{R}_{+}^{d}}\bigg(\sum_{i=1}^{d}z_{i}\bigg)\nu(dz)+\sum_{i=1}^{d}\bigg(x_{i}\int_{\mathbb{R}_{+}^{d}}\bigg(\sum_{j\neq i}^{d}z_{j}\bigg)\mu_{i}(dz)+\int_{\mathbb{R}_{+}^{d}}z_{i}\mu_{i}(dz)\bigg).
\]
So condition (2) can be verified in a similar way
as in the first case. 

We now prove (\ref{nonattainment}). Define $\tau_{n}:=\inf\{t\geq0:\min_{1\leq i\leq d}X_{i}(t)<\frac{1}{n}\}$,
$\sigma_{m}:=\inf\{t\ge0:\max_{1\leq i\leq d}X_{i}(t)>m\}$ and $\gamma_{m,n}:=\tau_{n}\wedge\sigma_{m}$.
By It$\hat{\text{o}}$'s formula,
\[
f(X(t\wedge\gamma_{m,n}))=f(x)+\int_{0}^{t\wedge\gamma_{m,n}}\mathcal{L}f(X(s))ds+M(t\wedge\gamma_{m,n}).
\]
Here $\left\{ M(t\wedge\gamma_{m,n})\right\} _{t\ge0}$ is the sum
of some stochastic integrals and is thus a local martingale, that
is, there are stopping times $(\gamma_{k})_{k\geq1}$ such that $\gamma_{k}\to+\infty$
a.s. and $M(t\wedge\gamma_{m,n}\wedge\gamma_{k})$ is a martingale.
By condition (2) we have
\[
\mathbb{E}[f(X(t\wedge\gamma_{m,n}\wedge\gamma_{k}))]=f(x)+\mathbb{E}[\int_{0}^{t\wedge\gamma_{m,n}\wedge\gamma_{k}}\mathcal{L}f(X(s))ds]\leq f(x)+d_{m}\mathbb{E}[\int_{0}^{t}f(X(s\wedge\gamma_{m,n}\wedge\gamma_{k}))ds].
\]
By Gronwall's inequality and letting $k\to+\infty$, we obtain $\mathbb{E}[f(X(t\wedge\gamma_{m,n}))]\leq f(x)e^{d_{m}t}$,  $t\geq0.$
Using condition (1) and letting $n\to\infty$, we get $\mathbb{P}(\tau_{0}>t\wedge\sigma_{m})=1$
for each $m\geq1$ and $t>0$, where $\tau_{0}:=\inf\{t>0:X(t)\in\partial\mathbb{R}_{+}^{d}\}$.
Since $X(t)$ is non-explosive, (\ref{nonattainment}) follows by letting $m\to\infty$ and then $t\to\infty$.
\end{proof}

\subsection{Boundary attainment}

In this section, we provide some sufficient conditions for a CIMBI
process to hit the boundary $\partial\mathbb{R}_{+}^{d}$.

We first consider the diffusion case. The proof of the following theorem
is based on a comparison argument due to \citep{cattiaux2010competitive}.
For this comparison to work, we need to assume that $b_{ij}=0$, $i\neq j$.
\begin{Theorem}\label{attainment behavior} Consider the following
$d$-dimensional CIMBI process: for $i=1,\ldots,d$,
\begin{equation}
X_{i}(t)=x_{i}+\int_{0}^{t}\bigg(\eta_{i}+b_{ii}X_{i}(s)+\gamma_{i}\left(X(s)\right)\bigg)ds+\int_{0}^{t}\sqrt{2\sigma_{i}X_{i}(s)}dW_{i}(s),\label{cimbi diffusion}
\end{equation}
where $x_{i}>0$ and $\sigma_{i}>0$. Assume that $\eta_{i}<\sigma_{i}/2$,
$i=1,...,d$. Define $\tilde{c}_{ii}:=c_{ii}$ and for
$i\neq j$, $\tilde{c}_{ij}:=\max(0,c_{ij})$. Then $\mathbb{P}[X(t)\in\partial\mathbb{R}_{+}^{d}\enspace\mbox{for some }\enspace t>0]=1$,
provided that one of the following two conditions holds:
\begin{enumerate}[{(1)}]
\item for each $i=1,\ldots,d$, $b_{ii}<0$ and $\sum_{i,j=1}^{d}\tilde{c}_{ij}\sigma_{j}y_{i}y_{j}\le0$,
$y\in\mathbb{R}_{+}^{d}$;
\item the quadratic form $\sum_{i,j=1}^{d}\tilde{c}_{ij}\sigma_{j}y_{i}y_{j}$,
$y\in\mathbb{R}^{d}$ is negative definite.
\end{enumerate}
\end{Theorem}
\begin{proof}
Set $Z_{i}(t)=2\sqrt{X_{i}(t)/2\sigma_{i}}$, by It$\hat{\text{o}}$'s
formula, we have
\begin{align*}
dZ_{i}(t)=dW_{i}(t)+\bigg(\frac{\eta_{i}}{\sigma_{i}Z_{i}(t)}-\frac{1}{2Z_{i}(t)}+\frac{b_{ii}Z_{i}(t)}{2}+\sum_{j=1}^{d}\frac{c_{ij}\sigma_{j}Z_{i}(t)(Z_{j}(t))^{2}}{4}\bigg)dt,
\end{align*}
$Z_{i}(0)=2\sqrt{x_{i}/2\sigma_{i}}$. Define $\sigma:=\inf\{t>0:X(t)\in\partial\mathbb{R}_{+}^{d}\}$.

Assume that condition (1) or (2) is true. Then $Z_{i}(t)\leq\tilde{U}_{i}(t)$
for $t<\sigma$, where
\[
d\tilde{U}_{i}(t)=dW_{i}(t)+\bigg(\frac{b_{ii}\tilde{U}_{i}(t)}{2}+\sum_{j=1}^{d}\frac{\tilde{c}_{ij}\sigma_{j}\tilde{U}_{i}(t)(\tilde{U}_{j}(t))^{2}}{4}\bigg)dt,\quad\tilde{U}_{i}(0)=Z_{i}(0).
\]
In fact, define $V_{i}(s):=Z_{i}(s)-\tilde{U}_{i}(s)$, then $V_{i}(\cdot)$
is $C^{1}$, $V_{i}(0)=0$ and
\begin{align*}
\frac{d}{ds}(V_{i}(s))\big|_{s=0}=\frac{\eta_{i}}{\sigma_{i}x_{i}}-\frac{1}{2x_{i}}+\sum_{j\neq i}\big(c_{ij}-\tilde{c}_{ij}\big)x_{i}x_{j}<0.
\end{align*}
So $V_{i}(s)<0$ for very small $s\in(0,t]$. Suppose that for some
$s\in(0,t]$, $V_{j}(s)=0$ for some $j$. Let $u:=\inf\left\{ s\in(0,t]:V_{j}(s)=0\mbox{ for some }j\right\} $.
Then $V_{i_{0}}(u)=0$ for some $i_{0}\in\{1,...,d\}$. So $Z_{i_{0}}(u)=\tilde{U}_{i_{0}}(u)$
and $Z_{j}(u)\leq\tilde{U}_{j}(u)$ for $j\neq i_{0}$. Moreover,
\[
\frac{d}{ds}(V_{i_{0}}(s))\big|_{s=u}=\frac{\eta_{i}}{\sigma_{i}Z_{i_{0}}(u)}-\frac{1}{2Z_{i_{0}}(u)}+\frac{1}{4}\sum_{j\neq i_{0}}\bigg(c_{i_{0}j}\sigma_{j}Z_{i_{0}}(u)(Z_{j}(u))^{2}-\tilde{c}_{i_{0}j}\sigma_{j}\tilde{U}_{i_{0}}(u)(\tilde{U}_{j}(u))^{2}\bigg)\leq\frac{\eta_{i}}{\sigma_{i}Z_{i_{0}}(u)}-\frac{1}{2Z_{i_{0}}(u)}<0,
\]
contradicting the fact that $V_{i_{0}}(s)<0,\ 0<s<u$ and $V_{i_{0}}(u)=0$.
Hence $Z(s)\leq\tilde{U}(s),\ 0\le s\le t$. As in the proof of Theorem
\ref{pathwise unique strong solution}, we can find constants $C_{1},C_{2}>0$
such that $\mathbb{E}[1+|\tilde{U}(t)|^{2}]\le C_{1}\exp(C_{2}t)$,
$t\ge0$, so $\tilde{U}(t)$ is
non-explosive. If we can prove that $\tilde{U}(t)$
hits the boundary $\partial\mathbb{R}_{+}^{d}$ with probability one,
then $Z_{i}(t)$ and $X_{i}(t)$ will also hit $\partial\mathbb{R}_{+}^{d}$
almost surely. To obtain this, it suffices to prove that $\tilde{U}(t)$
is Harris recurrent in $\mathbb{R}^{d}$ (see \citet[Section 2.2]{1993Stability}
for a definition), which implies $\tilde{U}_{i}(t)<0$
for some $t=t(\omega)>0$ and thus $\tilde{U}_{i}(s)=0$ for some
$s<t$ by continuity of paths. According to \citep[Theorem 3.3]{MR1234295},
irreducibility, T-process property, and existence of a Lyapunov function
imply the Harris recurrence. Now we show that these three properties
are met for the process $\tilde{U}(t)$.

``Lyapunov function'': It is not hard to verify that $f(x)=1+\sum_{i=1}^{d}x_{i}^{2}$
satisfies the condition (CD1) in \citep[Theorem 3.3]{MR1234295},
namely, $\mathcal{A}f(x)=d+\sum_{i=1}^{d}b_{ii}x_{i}^{2}+\sum_{i,j=1}^{d}\tilde{c} _{ij}\sigma_{j}x_{i}^{2}x_{j}^{2}/2\leq k\mathbbm{1}_{\{x:\Vert x\Vert\leq r\}}$,
$x\in\mathbb{R}^{d}$, where $\mathcal{A}$ is the generator of the
process $\tilde{U}(t)$ and $k$
and $r$ are positive constants.

`` T-process'': Note that $\tilde{U}(t)$ has the identity matrix as its
diffusion matrix and its
drift is continuous and thus bounded on compact sets. In addition, $\tilde{U}(t)$ is
non-explosive. So {[}3, Lemma 2.5{]} applies and $\tilde{U}(t)$ has
the strong Feller property. Its transition function is thus its own
continuous component (see \citep[page 495]{1993Stability} for a definition).
It follows from the definition on page 496 of \citep{1993Stability}
that $\tilde{U}(t)$ is a T-process.

``Irreducibility'': Let $\lambda$ be the Lebesgue measure on $\mathbb{R}^{d}$.
Suppose $\tilde{U}(0)=x$ and
$A\subset\mathbb{R}^{d}$ with $\lambda(A)>0$. Let $D_{R}:=\{x\in\mathbb{R}^{d}\,\big|\,|x|<R\}$,
where $R>0$ is large enough such that $x\in D_{R}$ and $\lambda(A\cap D_{R})>0$.
According to \citep[Theorem 4.2]{MR2247841} (see also \citep[Theorem 1.1]{MR2853532}),
the process $\tilde{U}$ started at $x$ and killed upon
exiting from $D_{R}$, denoted by $\tilde{U}_{x}^{R}$,
has a positive density $p^{R}(t,x,y)$ with respect to the Lebesgue
measure. Namely, defining $\tau:=\inf\{t>0:\tilde{U}(t)\notin D_{R}\}$,
it holds that $\mathbb{P}[\tilde{U}(t)\in E,t<\tau]=\int_{E}p^{R}(t,x,y)dy$
for $E\in\mathcal{B}(D_{R})$, where
\[
p^{R}(t,x,y)>0,\quad(t,x,y)\in(0,\infty)\times D_{R}\times D_{R}.
\]
Therefore, for any $t>0$, $\mathbb{P}[\tilde{U}(t)\in A]\ge\mathbb{P}[\tilde{U}(t)\in A\cap D_{R},t<\tau]=\int_{A\cap D_{R}}p^{R}(t,x,y)dy>0$.
Hence, $\tilde{U}(t)$ is irreducible
with respect to the Lebesgue measure.

So $\tilde{U}(t)$ is Harris
recurrent in $\mathbb{R}^{d}$ under condition (1) or (2). The assertion
is proved.
\end{proof}
It's worth noting that when $d=2$, our Theorem \ref{attainment behavior}
improves \citep[Theorem 2.2]{cattiaux2010competitive} in two ways:
1). we allow interactions that are of a mixed type; 2). even for the
cooperative interaction case, our conditions are significantly weaker
than their ``balance condition'' (see \citep[Equation (2.5)]{cattiaux2010competitive}).

\begin{Corollary} Consider the $d$-dimensional CIMBI process $X(t)$
given by \eqref{cimbi diffusion} and assume the same assumptions
as in Theorem \ref{attainment behavior} and additionally $\eta_{i}=0$
for each $i\in\{1,...,d\}$. Then $X(t)$ goes to extinction in finite
time with probability one. \end{Corollary}
\begin{proof}
Define $\sigma:=\inf\{t\ge0:X(t)\in\partial\mathbb{R}_{+}^{d}\}$.
We know from the previous theorem that $\sigma<\infty$ a.s. Since
$\eta_{i}=0$, $i=1,\ldots,d$, if $X(t)$ hits the boundary, then
there exists $i_{0}\in\{1,...,d\}$ such that $X_{i_{0}}$ goes to
extinction in finite time. Conditioning on $\left\{ X_{i_{0}}(\sigma)=0\right\} $,
the model, after $\sigma$, becomes a $(d-1)$-dimensional CIMBI process.
Then we can apply the same procedure to reduce it to a single-type
competition model that goes to extinction in finite time by \citet[Theorem 3.5]{MR2134113}.
\end{proof}
\begin{Theorem} \label{attainment behavior 2} Consider the following
$d$-dimensional CIMBI process: for $i=1,\ldots,d$,
\begin{equation}
\begin{aligned}X_{i}(t)= & x_{i}+\int_{0}^{t}\bigg(\eta_{i}+b_{ii}X_{i}(s)+\gamma_{i}\left(X(s)\right)\bigg)ds+\int_{0}^{t}\sqrt{2\sigma_{i}X_{i}(s)}dW_{i}(s)+\int_{0}^{t}\int_{\mathbb{R}_{+}^{d}}z_{i}N_{0}\left(ds,dz\right)\\
 & +\int_{0}^{t}\int_{\mathbb{R}_{+}}\int_{\mathbb{R}_{+}^{d}}z_{i}\mathbbm{1}_{\{u\leq X_{i}(s-)\}}\tilde{N}_{i}\left(ds,du,dz\right)+\sum_{j\neq i}\int_{0}^{t}\int_{\mathbb{R}_{+}}\int_{\mathbb{R}_{+}^{d}}z_{i}\mathbbm{1}_{\{u\leq X_{j}(s-)\}}N_{j}\left(ds,du,dz\right),
\end{aligned}
\label{ndimCBIC2}
\end{equation}
where $x_{i}>0$ and $\sigma_{i}>0$. Assume that $\nu,\mu_{1},\ldots,\mu_{d}$
are finite measures, $\eta_{i}<\sigma_{i}/2$, $i=1,...,d$. Define $\tilde{c}_{ii}:=c_{ii}$ and for
$i\neq j$, $\tilde{c}_{ij}:=\max(0,c_{ij})$. Then
$\mathbb{P}[X(t)\in\partial\mathbb{R}_{+}^{d}\enspace\mbox{for some }\enspace t>0]>0$,
provided that one the following two conditions is true:
\begin{enumerate}[{(1)}]
\item for each $i=1,\ldots,d$, $b_{ii}-\int_{\mathbb{R}_{+}^{d}}z_{i}\mu_{i}(dz)<0$
and $\sum_{i,j=1}^{d}\tilde{c}_{ij}\sigma_{j}y_{i}y_{j}\le0$, $y\in\mathbb{R}_{+}^{d}$;
\item the quadratic form $\sum_{i,j=1}^{d}\tilde{c}_{ij}\sigma_{j}y_{i}y_{j}$,
$y\in\mathbb{R}^{d}$ is negative definite.
\end{enumerate}
\end{Theorem}
\begin{proof}
We first consider the following equation: for $i=1,\ldots,d$,

\begin{align}
Y_{i}(t)=x_{i}+\int_{0}^{t}\bigg(\eta_{i}+\tilde{b}_{ii}Y_{i}(s)+\gamma_{i}(Y(s))\bigg)ds+\int_{0}^{t}\sqrt{2\sigma_{i}Y_{i}(s)}dW_{i}(s),\label{nonattainment of CBII diffusion }
\end{align}
where $\tilde{b}_{ii}=b_{ii}-\int_{\mathbb{R}_{+}^{d}}z_{i}\mu_{i}(dz)$,
and for $m\in\mathbb{N}$, define
\begin{align}
Y_{m,i}(t)=x_{i}+\int_{0}^{t}\bigg(\eta_{i}+\tilde{b}_{ii}Y_{m,i}(s)\wedge m+\gamma_{i}(Y_{m}(s)\wedge m)\bigg)ds+\int_{0}^{t}\sqrt{2\sigma_{i}Y_{m,i}(s)\wedge m}dW_{i}(s).\label{nonattainment of CBII truncated diffusion}
\end{align}
Define $\tau_{m}:=\inf\{t\geq0:\max_{1\leq i\leq d}Y_{m,i}(t)>m\}$.
As shown in the proof of Theorem \ref{pathwise unique strong solution},
both equations (\ref{nonattainment of CBII diffusion }) and (\ref{nonattainment of CBII truncated diffusion})
have a pathwise unique strong solution and $Y_{m}(t)=Y(t)$ on $[0,\tau_{m})$.
By interlacing argument used in \citep[Theorem 4.9.1]{MR1011252},
the following equation also has a pathwise unique strong solution:
for $i=1,\ldots,d$,
\begin{align*}
Y_{m,i}^{\prime}(t) & =x_{i}+\int_{0}^{t}\bigg(\eta_{i}+\tilde{b}_{ii}Y_{m,i}^{\prime}(s)\wedge m+\gamma_{i}(Y_{m}^{\prime}(s)\wedge m)\bigg)ds+\int_{0}^{t}\sqrt{2\sigma_{i}Y_{m,i}^{\prime}(s)\wedge m}dW_{i}(s)\\
 & +\int_{0}^{t}\int_{\mathbb{R}_{+}^{d}}z_{i}N_{0}(ds,dz)+\sum_{j=1}^{d}\int_{0}^{t}\int_{\mathbb{R}_{+}}\int_{\mathbb{R}_{+}^{d}}z_{i}\mathbbm{1}_{\{u\leq Y_{m,j}^{\prime}(s-)\wedge m\}}N_{j}\left(ds,du,dz\right).
\end{align*}
Set $\tau_{m}^{\prime}:=\inf\{t\geq0:\max_{1\leq i\leq d}Y_{m,i}^{\prime}(t)>m\}$,
then $Y_{m}^{\prime}(t)=X(t)$ on $[0,\tau_{m}^{\prime})$, by pathwise
uniqueness and the gluing argument in \citep[Proposition 2.4]{Fuzongfei2010spa}.

Define $\sigma:=\inf\{t\geq0:Y(t)\in\partial\mathbb{R}_{+}^{d}\}$,
according to Theorem 3.2, we have $\mathbb{P}(\sigma<\infty)=1$.
Then there exists $T>0$ such that $\mathbb{P}(\sigma\leq T)>0$.
Since $\tau_{m}\nearrow+\infty$ as $m\to\infty$, there exists $M>0$
such that $\mathbb{P}(\sigma\leq T<\tau_{M})>0$.

Define $A:=\{\sigma\leq T<\tau_{M}\}$ and
\begin{equation}
B:=\bigg\{ N_{0}\big((0,T]\times\mathbb{R}_{+}^{d}\big)+\sum_{j=1}^{d}N_{j}\big((0,T]\times[0,M]\times\mathbb{R}_{+}^{d}\big)=0\bigg\},\label{eq: event B}
\end{equation}
where $\mathbb{P}(B)>0$ because $\nu,\mu_{1},\ldots,\mu_{d}$ are
finite measures. Notice that $W$ is independent of $N_{0},N_{1},\ldots,N_{d}$,
we see that $A$ and $B$ are independent. Therefore,
\[
\mathbb{P}(A\cap B)=\mathbb{P}(A)\mathbb{P}(B)>0.
\]
Next we prove that on $A\cap B$, $Y(t)=Y_{M}(t)=Y_{M}^{\prime}(t)=X(t)$
for $t\in[0,T]$.

On the event $A$, since $T<\tau_{M}$, we have $Y(t)=Y_{M}(t)$ for
$t\in[0,T]$. On the event $B$, in view of the interlacing construction
in the proof of \citep[Theorem 4.9.1]{MR1011252} and \eqref{eq: event B}
as well as the fact that compensator
effect of $\tilde{N}_{i}$ has already been absorbed into the modified
drift $\tilde{b}_{ii}=b_{ii}-\int_{\mathbb{R}_{+}^{d}}z_{i}\mu_{i}(dz)$,
we obtain $Y_{M}(t)=Y_{M}^{\prime}(t)$ for $t\in[0,T]$. Hence on
$A\cap B$, $Y_{M,i}^{\prime}(t)<M$ for $t\in[0,T]$, $i=1,\ldots,d$
and thus $T<\tau_{M}^{\prime}$, which implies that on $A\cap B$,
$Y_{M}^{\prime}(t)=X(t)$ for $t\in[0,T]$. So on $A\cap B$, we have
$Y(t)=Y_{M}(t)=Y_{M}^{\prime}(t)=X(t)$ for $t\in[0,T]$.

On $A\cap B$, since $\sigma\leq T$, we see that $X(t)=Y(t)\in\partial\mathbb{R}_{+}^{d}$
for some $t\le T$. The assertion is proved.
\end{proof}
We next show that if the interaction is competitive, then the condition
$\eta_{i}<\sigma_{i}/2$ in Theorems \ref{attainment behavior}
and \ref{attainment behavior 2} can be slightly relaxed.

\begin{Theorem} Consider the $d$-dimensional CIMBI process given
by \eqref{ndimCBIC2}, where $x_{i}>0$ and $\sigma_{i}>0$ for each
$i=1,\ldots,d$. Assume that $\eta_{i}<\sigma_{i}$ for each $i=1,\ldots,d$
and the interaction is competitive, that is, $c_{ij}\le0$ for all
$i\neq j$. If $\nu,\mu_{1},\ldots,\mu_{d}$ are finite measures,
then $\mathbb{P}[X(t)\in\partial\mathbb{R}_{+}^{d}\enspace\mbox{for some }\enspace t>0]>0$.
Moreover, if $\nu,\mu_{1},\ldots,\mu_{d}$ are all zero, then $\mathbb{P}[X(t)\in\partial\mathbb{R}_{+}^{d}\enspace\mbox{for some }\enspace t>0]=1$.
\end{Theorem}
\begin{proof}
It suffices to prove that $X(t)$ hits the boundary $\partial\mathbb{R}_{+}^{d}$
almost surely when $\nu,\mu_{1},\ldots,\mu_{d}$ are all zero, since
the argument in the proof of Theorem 3.4 applies here and the rest
of the assertions follow.

Let $Y(t)$ be a single type CBI process with competition given by
\[
Y_{i}(t)=x_{i}+\int_{0}^{t}\bigg(\eta_{i}+b_{ii}Y_{i}(s)+c_{ii}(Y_{i}(s))^{2}\bigg)ds+\int_{0}^{t}\sqrt{2\sigma_{i}Y_{i}(s)}dW_{i},\quad t\ge0,\,i=1,\ldots,d.
\]
Then $X_{i}(t)\leq Y_{i}(t)$, $t\ge0,\,i=1,\ldots,d$. This comparison
principle can be obtained using the same idea as in \citep[Proposition 4.2]{MR3208120},
where a two-type CBI process with intra-specific competition is shown
to be dominated by a two-type CBI; see also \citep[Theorem 2.2]{MR4195178}.
In fact, the proofs in \citet{MR3208120,MR4195178} work here with
obvious adaptations due to the fact that $c_{ij}\leq0$ for $i\neq j$.

According to \citep[Remark 9, page 44]{MR3496029}, $Y_{i}(t)$ hits
zero almost surely if $\eta_{i}<\sigma_{i}$, so $X(t)$ hits the
boundary $\partial\mathbb{R}_{+}^{d}$ in finite time with probability
one. The assertion is proved.
\end{proof}

\section*{Acknowledgment}

We would  like to thank the anonymous referee for his/her valuable comments and suggestions, particularly for pointing out one error in Theorem \ref{attainment behavior} of the first version of this work.  Our work was supported by the Guangdong Provincial Key Laboratory
of IRADS (2022B1212010006, open project code: UICR0600008-2).




\end{document}